\documentclass[a4paper,10pt]{amsart}

\usepackage{amsfonts}

\usepackage{graphicx}
\usepackage{amsmath}
\usepackage{amssymb}
\usepackage[all]{xypic}
\usepackage[all]{xy}

\begin{document}
\input xy
\xyoption{all}

\renewcommand{\mod}{\operatorname{mod}\nolimits}
\newcommand{\proj}{\operatorname{proj}\nolimits}
\newcommand{\rad}{\operatorname{rad}\nolimits}
\newcommand{\soc}{\operatorname{soc}\nolimits}
\newcommand{\ind}{\operatorname{ind}\nolimits}
\newcommand{\Top}{\operatorname{top}\nolimits}
\newcommand{\ann}{\operatorname{Ann}\nolimits}
\newcommand{\id}{\operatorname{id}\nolimits}
\newcommand{\Mod}{\operatorname{Mod}\nolimits}
\newcommand{\End}{\operatorname{End}\nolimits}
\newcommand{\Ob}{\operatorname{Ob}\nolimits}
\newcommand{\Ht}{\operatorname{Ht}\nolimits}
\newcommand{\cone}{\operatorname{cone}\nolimits}
\newcommand{\rep}{\operatorname{rep}\nolimits}
\newcommand{\Ext}{\operatorname{Ext}\nolimits}
\newcommand{\Hom}{\operatorname{Hom}\nolimits}
\renewcommand{\Im}{\operatorname{Im}\nolimits}
\newcommand{\Ker}{\operatorname{Ker}\nolimits}
\newcommand{\Coker}{\operatorname{Coker}\nolimits}
\renewcommand{\dim}{\operatorname{dim}\nolimits}
\newcommand{\Ab}{{\operatorname{Ab}\nolimits}}
\newcommand{\Coim}{{\operatorname{Coim}\nolimits}}
\newcommand{\pd}{\operatorname{pd}\nolimits}
\newcommand{\sdim}{\operatorname{sdim}\nolimits}
\newcommand{\add}{\operatorname{add}\nolimits}
\newcommand{\cc}{{\mathcal C}}

\newtheorem{theorem}{Theorem}[section]
\newtheorem{acknowledgement}[theorem]{Acknowledgement}
\newtheorem{algorithm}[theorem]{Algorithm}
\newtheorem{axiom}[theorem]{Axiom}
\newtheorem{case}[theorem]{Case}
\newtheorem{claim}[theorem]{Claim}
\newtheorem{conclusion}[theorem]{Conclusion}
\newtheorem{condition}[theorem]{Condition}
\newtheorem{conjecture}[theorem]{Conjecture}
\newtheorem{corollary}[theorem]{Corollary}
\newtheorem{criterion}[theorem]{Criterion}
\newtheorem{definition}[theorem]{Definition}
\newtheorem{example}[theorem]{Example}
\newtheorem{exercise}[theorem]{Exercise}
\newtheorem{lemma}[theorem]{Lemma}
\newtheorem{notation}[theorem]{Notation}
\newtheorem{problem}[theorem]{Problem}
\newtheorem{proposition}[theorem]{Proposition}
\newtheorem{remark}[theorem]{Remark}
\newtheorem{solution}[theorem]{Solution}
\newtheorem{summary}[theorem]{Summary}
\newtheorem*{thma}{Theorem}

\title[Lifting to cluster-tilting objects in 2-CY triangulated categories]
{Lifting to cluster-tilting objects in 2-Calabi-Yau triangulated
categories}

\author[Fu and Liu]{Changjian Fu and Pin Liu }
\address{ Department of Mathematics\\
Sichuan University\\
 610064 Chengdu \\
P.R.China } \email{
\begin{minipage}[t]{5cm}
flyinudream@yahoo.com.cn
\end{minipage}
}
\address{
 Department of Mathematics\\
   Sichuan University\\
  610064 Chengdu \\
   P.R.China}
\email{
\begin{minipage}[t]{5cm}
pinliu@yahoo.cn \\
\end{minipage}
}
\date{last modified on December 29, 2007}
\subjclass{18E30, 16D90}
\keywords{2-Calabi-Yau category, tilting
modules, cluster-tilting objects}

\begin{abstract} We show that a tilting module over the endomorphism
algebra of a cluster-tilting object in a 2-Calabi-Yau triangulated
category lifts to a cluster-tilting object in this 2-Calabi-Yau
triangulated category. This generalizes a recent work of D. Smith
for cluster categories.
\end{abstract}

\maketitle

\section{Introduction}
Let $k$ be an algebraically closed field and $H$ be a
finite-dimensional hereditary algebra. The associated cluster
category $\mathcal{C}_H$ was introduced and studied in \cite{BMRRT},
and also in \cite{CCS} for algebras $H$ of Dynkin type $A_n$. This
is a certain triangulated category \cite{Kel} which was invented in
order to model some ingredients in the definition of cluster
algebras introduced and studied by Fomin-Zelevinsky and
Berenstein-Fomin-Zelevinsky in a series of articles \cite{FZ1, FZ2,
BFZ, FZ3}. For this purpose, a tilting theory was developed in the
cluster category. This further led to the theory of cluster-tilted
algebras initiated in \cite{BMR}.

In \cite{GLS2}, C. Gei{\ss}, B. Leclerc and J. Schr{\"o}er have
shown that the module category of a preprojective algebra of Dynkin
type is connected to cluster algebras in a way similar to the
connection between cluster categories and cluster algebras.  The
preprojective algebra approach is particularly suited for cluster
algebras associated to algebraic group, as it is not necessary to
start with a finite-dimensional hereditary algebra. Some special
modules over a preprojective algebra, treated as cluster-tilting
objects in cluster categories, are called maximal rigid modules.

All these representation-theoretic approaches to cluster algebras
are now called "categorifying" cluster algebras: the cluster-tilting
objects play the role of clusters and their indecomposable direct
summands the one of the cluster variables. Both cluster categories
and stable module categories of preprojective algebras of Dynkin
type are Calabi-Yau triangulated category of CY-dimension 2. This
motivates the study of  2-Calabi-Yau  categories  in \cite{KR1,
BIRS, GLS3} and so on.

 At the same time, cluster-tilted algebras
have been further developed in many papers by several authors, and
revealed to have very nice properties, see for instance \cite{BMR,
KR1, S}. In \cite{S}, D. Smith gave an answer to how to identify
tilting modules over cluster-tilted algebras. He showed that a
tilting modules over a cluster-tilted algebra can be lifted to a
cluster-tilting object in the  cluster category.

In this note, we point out that this phenomenon does not depend on
cluster category but only on the 2-CY property. Namely, we prove the
following

\begin{thma}Let $\mathcal {C}$ be a Calabi-Yau triangulated category of
CY-dimension 2 with a cluster-tilting object $T$ and let $\Gamma$ be
the endomorphism algebra of $T$. Let $L$ be a tilting module over
$\Gamma$, then $L$ lifts to a cluster-tilting object in $\mathcal
{C}$.
\end{thma}

\vspace{0.2cm} \noindent{\bf Acknowledgments.} The authors would
like to thank David Smith for many useful comments. They are
grateful to Bernhard Keller to point out a gap in a previous version
of this note. They thank Liangang Peng for helpful discussions. They
are grateful to Idun Reiten and Bin Zhu for their interest.

\section{Preliminaries}
In this section we review some useful notations and results.
\subsection{Tilting modules} Let $k$ be an algebraically closed field and $A$ be a finite-dimensional algebra.
Let $\mod A$ be
the category of finite-dimensional right $A$-modules.
 For an $A$-module
$T$, let $\add T$ denote the full subcategory of $\mod A$ with
objects all direct summands of direct sums of copies of $T$. Then
$T$ is called a tilting module in $\mod A$ if
\begin{enumerate}
\item[-]$\pd_AT\leq 1$,
\item[-]$\Ext^1_A(T,T)=0$,
\item[-]there is an exact sequence $0\to A\to T^0\to T^1\to 0$, with
$T^0, T^1$ in $\add T$.
\end{enumerate}
This is the original definition of tilting modules from \cite{HR},
and it was proved in \cite{B} that the third axiom can be replaced
by the following:
\begin{enumerate}
\item[-]the number of indecomposable direct summands of $T$ (up to
isomorphism) is the same as the number of  simple $A$-modules.
\end{enumerate}

\subsection{Cluster categories and cluster-tilted
algebras}\label{section} Let $H$ be a hereditary algebra. The
cluster category $\mathcal{C}_H$ is the orbit category
$D^b(H)/\tau^{-1}S$, where $S$ denotes the suspension functor and
$\tau$ is the Auslander-Reiten translation in the bounded derived
category $D^b(H)$. It is shown in \cite{Kel} that $\mathcal{C}_H$ is
a triangulated category. In the cluster categories approach to
cluster algebras, a central role is played by the following
2-Calabi-Yau property (see \cite{BMRRT}). For any $X, Y$ in
$\mathcal{C}_H$,
$$D\Ext^1_{\mathcal{C}_H}(X,Y)\simeq\Ext^1_{\mathcal{C}_H}(Y,X).$$

 A cluster-tilted algebra
$\Lambda$ is the endomorphism algebra of a cluster-tilting object
$T$ in $\mathcal{C}_H$. This means that $T$ has no self-extensions,
and any direct sum $T\oplus T'$ with an indecomposable $T'$ not
occuring as a direct summand of $T$ does have self-extensions (see
more in \cite{BMRRT}). By \cite{BMR}, the functor
$\Hom_{\mathcal{C}_H}(T,-)$ induces an equivalence
$\mathcal{C}_H/\add \tau T\to \mod \Lambda$. Moreover the
cluster-tilted algebra $\Lambda$ has the same number of
non-isomorphic indecomposable modules as $H$.

Recently, D. Smith proved in \cite{S} the following.
\begin{theorem}Let $\mathcal{C}_H$ be a cluster category, $T$ be a
cluster-tilting object in $\mathcal{C}_H$ and let
$\Lambda=\End_{\mathcal{C}_H}(T)^{op}$ be the corresponding
cluster-tilted algebra. Then the tilting $\Lambda$-modules lift to
cluster-tilting objects in $\mathcal{C}_H$.
\end{theorem}

\subsection{Calabi-Yau triangulated categories of CY-dimension 2}
 Let
$k$ be an algebraically closed field and $\mathcal {C}$ be a
Krull-Schmidt triangulated $k$-linear category with split
idempotents and suspension functor $S$. We suppose that all
$\Hom$-space of $\mathcal {C}$ are finite-dimensional and that
$\mathcal {C}$ admits a Serre functor $\Sigma$. We suppose that
$\mathcal {C}$ is Calabi-Yau of CY-dimension 2, i.e. there is an
isomorphism of triangle functors
$$S^2\stackrel{\sim}\rightarrow \Sigma.$$For $X,Y\in \mathcal {C}$ and
$n\in \mathbb{Z}$, we put as usual $$\Ext_\mathcal
{C}^n(X,Y)=\Hom_\mathcal {C}(X,S^nY).$$Thus the Calabi-Yau property
can be writen as the following  bifunctorial isomorphisms
$$D\Ext^1(X,Y)\simeq \Ext^1(Y,X),\ \text{for any}\ X, Y.$$This
setting holds not only for cluster categories, for stable module
categories of preprojective algebras of Dynkin type, but also for
certain stable Frobenius subcategory of the category of all
finite-dimensional nilpotent representations of preprojective
algebras which are not Dynkin type \cite{GLS3},
 and also for stable categories of
Cohen-Macaulay modules over commutative complete local Gorenstein
isolated singularities of dimension 3 (see more in \cite{BIRS}).

Cluster-tilting objects in Calabi-Yau triangulated categories of
CY-dimension 2 (or 2-Calabi-Yau triangulated category for short) and
the corresponding endomorphism
 algebras were defined and studied
first in \cite{KR1}. In this note, we suppose that $\mathcal {C}$
always admits some cluster-tilting object $T$, which means
\begin{enumerate}
\item[-]$\Ext^1_\mathcal {C}(T,T)=0$ and
\item[-] for any $X\in \mathcal {C}$, if $\Hom_\mathcal
{C}(X,ST)=0,$ then $X$ belongs to $ \add T$, the full subcategory
formed by the direct factors of finite direct sums of copies of $T$
in $\mathcal{C}$.
\end{enumerate}
Note that because of the Calabi-Yau property, the second condition
is self dual. And by \cite{BMRRT}, this definition of
cluster-tilting objects and the definition in section \ref{section}
coincide for cluster categories. It is known from \cite{KR1} that
for each $X\in \mathcal {C}$, there exists a minimal right $\add T
$-approximation $T_0^X\xrightarrow{\delta} X$ and an induced
triangle
$$T_1^X\to T_0^X\xrightarrow{\delta} X\to ST_1^X,$$ where $T_0^X, T_1^X \in \add
T$. Sometimes, we call this triangle "minimal approximation
triangle" for convenience.

There is an essential result in \cite{KR1} describing the
relationship between 2-CY triangulated categories and the
corresponding endomorphism algebra of a cluster-tilting object,
which is the following.
\begin{theorem}Let $\mathcal {C}$ be a Calabi-Yau triangulated category of
CY-dimension 2 with a cluster-tilting object $T$ and let $\Gamma$ be
the endomorphism algebra of $T$. Then the functor $F=\Hom_\mathcal
{C}(T,-): \mathcal {C}\rightarrow \mod \Gamma$ induces an
equivalence $\mathcal {C}/\add ST\stackrel{\sim}\rightarrow
\mod\Gamma $. Moreover, the functor $F$ induces an equivalence from
$\add T$ to the category of projective modules in $\mod \Gamma$.
\end{theorem}

By \cite{RV},  $\mathcal {C}$ has Auslander-Reiten triangles and
$\tau_\mathcal {C}$ is induced by $\Sigma\circ S^{-1} \simeq S$.
Furthermore the following proposition was also proved in \cite{KR1}.
\begin{proposition}\label{proposition}$\mathcal {C}/\add ST\stackrel{\sim}\rightarrow \mod
\Gamma$ has Auslander-Reiten sequences, induced by  the
Auslander-Reiten triangles in $\mathcal {C}$.
\end{proposition}

Moreover, it is showed in \cite{DK} that all cluster-tilting objects
have the same number of pairwise non-isomorphic indecomposable
direct summands.
\section{Proof of the main result}
We need some preparation to prove the theorem. First, we have the
following well-known Auslander-Reiten formula in module theory.
\begin{lemma}\label{lemma}Let $A$ be an algebra and $M$ be an
$A$-module. If $\pd_A M\leqslant 1,$ then $\Ext_A^1(M,N)\simeq
D\Hom_A(N,\tau M)$ for each $A$-module $N$, where $\tau$ is the
Auslander-Reiten translation and $D=\Hom_k(-,k)$ is the usual
duality.
\end{lemma}

Then we need the following crucial proposition.
\begin{proposition}\label{crucial proposition}Let $\mathcal {C}$ be a 2-Calabi-Yau triangulated category
 with a cluster-tilting object $T$ and let $\Gamma$ be
the endomorphism algebra of $T$. Let $M, N$ be two objects in
$\mathcal {C}$ and $FM, FN\in \mod\Gamma$ be their images under the
functor $F=\Hom_\mathcal{C}(T,-):\mathcal{C}\to \mod\Gamma$. If $FM$
and $FN$ are of projective dimension at most one and satisfy
$\Ext_\Gamma^1(FM,FN)=0$ and $\Ext_\Gamma^1(FN,FM)=0$, then
$\Ext_\mathcal {C}^1(M,N)=0$ and $\Ext_\mathcal {C}^1(N,M)=0$.
\end{proposition}
\begin{proof}We only need to show that the result holds  for $M,N$
indecomposable. Since  $FN$ is of projective dimension at most 1,
using Auslander-Reiten formula (Lemma \ref{lemma}) and Proposition
\ref{proposition},  we have
$$0=D\Ext_\Gamma^1(FN,FM)\simeq\Hom_\Gamma(FM,\tau_\Gamma FN)$$
$$\simeq \frac{\Hom_\mathcal {C}(M,SN)}{\{f:M\rightarrow SN\
\text{factoring through} \add ST\}}.$$Therefore, any map in
$\Hom_\mathcal {C}(M,SN)$ factors through $\add ST$. So it suffices
to prove that any $\alpha: M\rightarrow SN$ which factors through
$\add ST$ equals zero.

Let
\[N\xrightarrow{\gamma} E\xrightarrow{\beta}M\xrightarrow{\alpha}SN
\]be the induced triangle, and let \[T_1^M\to T_0^M\xrightarrow{p_0^M} M\xrightarrow{b} ST_1^M\]
be the minimal approximation triangle of $M$. Since $\alpha: M\to
SN$ factors through $\add ST$, the composition $\alpha \cdot p^M_0$
is zero. Then $\alpha$ factors through $b$. That is, there exists
$\xi: T_1^M\to N$ such that $\alpha= S\xi\cdot b$.
\[\xymatrix{T_1^M\ar[r] &T_0^M\ar[r]^{p_0^M} & M\ar[r]^b
\ar[d]_\alpha & ST_1^M \ar@{.>}[ld]^{S\xi}\\
& & SN  }\]

Applying the functor $F=\Hom_\mathcal{C}(T, -)$ to the triangle
\[S^-M\xrightarrow{-S^-b}T_1^M\xrightarrow{} T_0^M\xrightarrow{p_0^M} M,\]
 we get the exact sequence in $\mod \Gamma$
\[FS^-M\xrightarrow{-FS^-b} FT_1^M\to FT_0^M\to FM\to 0.\]
Note that $F$ induces an equivalence between $\add T$ and the
category of projectives in $\mod \Gamma$, we have $FS^-b= 0$ in
$\mod \Gamma$ for the reason that we are considering the minimal
$\add T$-approximation triangle and $FM$ is of projective dimension
at most 1.  In particular, for any morphism $T\to S^-M$, the
composition
\[T\to S^-M\xrightarrow{S^-b}T_1^M\xrightarrow{\xi}N
\]
is zero, {\em i.e.} $FS^-\alpha=0$.  Thus, applying the functor $F$
to the triangle
\[N\xrightarrow{\gamma} E\xrightarrow{\beta}M\xrightarrow{\alpha}SN,
\]we get an exact sequence
\[0\to FN\to FE\xrightarrow{F\beta} FM\to 0,\]
which splits since $\Ext_\Gamma^1(FM, FN)=0$. This means $F\beta$ is
an split epimorphism,  by the quotient property, there exists $\rho:
M\to E$ in $\mathcal{C}$
\[\xymatrix{
& & M\ar[d]^{1_M} \ar@{.>}[dl]_\rho\\
N\ar[r]& E\ar[r]^\beta & M\ar[r]^\alpha & SN }
\]such that $1_M= \beta\cdot \rho + f$ in $\mathcal{C}$ for some $f: M\to
M$ which factors through $\add ST$. Thus $\alpha=
\alpha\cdot\beta\cdot\rho+ \alpha \cdot f= \alpha\cdot f$. Note that
we have $FS^-\alpha=0$, which implies that $0=\Hom_\mathcal{ C}(T,
S^-\alpha)\simeq \Hom_\mathcal{C}(ST, \alpha)$. That is, the
composition $ST\to M\xrightarrow{\alpha} N[1]$ equals $0$. Thus we
have $\alpha=\alpha\cdot f=0$, which implies
$\Hom_\mathcal{C}(M,SN)=0$, that is $ \Ext_\mathcal{C}^1(M,N)=0$.
Thanks to the 2-CY property, we get $\Ext_\mathcal {C}^1(N,M)=0$,
too.
\end{proof}

We can prove our main result now.
\begin{theorem}\label{theorem}Let $\mathcal {C}$ be a Calabi-Yau triangulated category of
CY-dimension 2 with a cluster-tilting object $T$ and let $\Gamma$ be
the endomorphism algebra of $T$. Let $L$ be a tilting module over
$\Gamma$, then  $L$ lifts to a cluster-tilting object in $\mathcal
{C}$.
\end{theorem}

\begin{proof}Note that $F=\Hom_\cc(T,-): \mathcal {C}\to \mod \Gamma$ is
 dense, so the preimage of $L$, denoted by $L^+$, has no self-extension by
Proposition \ref{crucial proposition}. Suppose that $Y$ is an
indecomposable object of $\cc$ such that
\[
\Hom_\cc(L^+, SY)=0.
\]
We need to show that $Y$ is a direct factor of $L^+$.

First,  we claim that $Y$ can not be a direct factor of $ST$. If
not, suppose that $Y\simeq ST_1$ for some $T_1\in \add T$. We have
\[0=D\Hom_\cc(L^+, SY)\simeq D\Hom_\cc(L^+, S^2T_1)\simeq \Hom_\cc(T_1, L^+),
\]
which implies
\[\Hom_{\Gamma}(FT_1, L)=0.
\]
 Since $L$ is a tilting module, we have
an exact sequence
\[0\to \Gamma\to L^0\to L^1\to 0
\]
where $L^0$ and $L^1$ are direct factors of finite direct sums of
$L$. Note that $FT_1$ is a projective $\Gamma$-module. But applying
the functor $\Hom_{\Gamma}(FT_1, -)$ to the short exact sequence
above, one gets
\[\Hom_{\Gamma}(FT_1, \Gamma)=0,
\]
a contradiction.

By the Auslander-Reiten formula (Lemma \ref{lemma}), it is not hard
to check that $\Ext_\Gamma^1(FL^+, FY)=0$. Since $FL^+$ is a tilting
module, this implies that there exists a short exact sequence
\[
 L_1 \xrightarrow{f} L_0 \xrightarrow{g} FY \to 0,
\]
with $L_1,L_0\in \add L$. Let $K$ be the image of $f$, and rewrite
$f: L_1\to L_0$ as the composition
$L_1\xrightarrow{f}K\xrightarrow{i}L_0$. Denoted by $L_1^+, L_0^+$
and $ K^+$  the objects in $\cc$ corresponding to $L_1, L_0$ and $
K$ respectively. By \cite[Lemma 8]{Palu}, we have  the following
triangle
\[K^+\xrightarrow{(i^+,\ r^+)}L_0^+\oplus ST_K\xrightarrow{g^+}
Y\to SK^+
\]
such that its image under the functor $F$ is
\[0\to K\xrightarrow{i}L_0\xrightarrow{g}FY\to 0.
\]

 Let $L_1^+\xrightarrow{f^+}K^+$ be the morphism in $\cc$ corresponding to
 $L_1\xrightarrow{f}
K$ in $\mod \Gamma$. By the octahedral axiom, we have the following
commutative diagram.
\[\xymatrix{L_1^+\ar@{=}[r]\ar[d]^-{f^+}&L_1^+\ar[d]^-{(i^+f^+,\ r^+f^+)}\\
K^+\ar[r]^-{(i^+,\ r^+)}\ar[d]&L_0^+\oplus ST_K\ar[r]^-{g^+}\ar[d]^{\alpha}& Y\ar@{=}[d]\\
X\ar[r]&M\ar[r]^{\beta}&Y}
\]
Applying $F$ to the triangle in the second column, we get an exact
sequence in $\mod\Gamma$
\[L_1\xrightarrow{f}L_0\xrightarrow{F\alpha}FM\to F(SL_1^+).
\]
Thus we have the following commutative diagram whose rows are exact
sequences in $\mod\Gamma$
\[\xymatrix{L_1\ar[r]^f\ar@{=}[d]&L_0\ar[r]^{g}\ar@{=}[d]& FY\ar[r]\ar@{.>}[d]^{ \gamma}&0\\
L_1\ar[r]^f\ar@{=}[d]& L_0\ar[r]^{F\alpha}\ar@{=}[d]&FM\ar[r]\ar[d]^{F\beta}& FSL_1^+\\
L_1\ar[r]^f&L_0\ar[r]^{g}& FY\ar[r]&0 }
\]
which implies that $\gamma \cdot F\beta\simeq1_{FY}$. In particular,
we have $M\simeq Y\oplus Z$ in $\cc$ for some $Z$, since we have
proved that $Y\not \simeq ST'$ for any $T'\in \add T$. Thus we
rewrite the triangle as
\[L_1^+\xrightarrow{(i^+f^+,\ r^+f^+)} L_0^+\oplus
ST_K\xrightarrow{\alpha} Y\oplus Z\xrightarrow{(0, \eta)^t}SL_1^+,
\]
since $\Hom_\cc(Y, SL^+)=0$. Consider the following commutative
diagram whose rows are triangles in $\cc$.
\[\xymatrix{0\ar[r]\ar[d]&Y\ar@{=}[r]\ar@{.>}[d] & Y\ar[r]\ar[d]^{(1,0)} & 0\ar[d]\\
L_1^+\ar[r]\ar[d]& L_0^+\oplus ST_K\ar[r]\ar@{.>}[d]& Y\oplus
Z\ar[r]^-{(0, \eta)^t}\ar[d]^{(1,0)^t}&
SL_1^+\ar[d]\\
0\ar[r]&Y\ar@{=}[r]&Y\ar[r]& 0}
\]
It yields that $Y$ is a direct factor of $L_0^+\oplus ST_K$,  which
implies that $Y$ is a direct factor of $L_0^+$, since that $Y$ is
not a direct factor of $ST$. This means that $L^+$ is a
cluster-tilting object in $\cc$.
\end{proof}

By the definition of quotient category, we can use the same notation
for a $\Gamma$-module and its preimage in $\mathcal{C}$ under the
projection $\mathcal{C}\to \mathcal{C}/\add ST
\xrightarrow{\sim}\mod\Gamma$.
 Thus we have the following corollary as for the cluster
categories \cite[Corollary 2.4]{S}.
\begin{corollary}Let $\mathcal {C}$ be a Calabi-Yau category of
CY-dimension 2 with a cluster-tilting object $T$ and let $\Gamma$ be
the endomorphism algebra of $T$ in $\mathcal{C}$. If $L$ is a
tilting $\Gamma$-module, then the endomorphism algebra
$\End_\Gamma(L)^{op}$ is a quotient of $\End_\mathcal{C}(L)^{op}$,
the endomorphism algebra of $L$ in $\mathcal{C}$.
\end{corollary}

\begin{proof}It is known that $L$ is actually a cluster-tilting object in
$\mathcal{C}$ by the Theorem \ref{theorem}. Thus the result is
deduced by using the equivalence $\mathcal{C}/\add ST
\xrightarrow{\sim} \mod\Gamma$.
\end{proof}

\def\cprime{$'$}
\providecommand{\bysame}{\leavevmode\hbox
to3em{\hrulefill}\thinspace}
\providecommand{\MR}{\relax\ifhmode\unskip\space\fi MR }
\providecommand{\MRhref}[2]{%
  \href{http://www.ams.org/mathscinet-getitem?mr=#1}{#2}
} \providecommand{\href}[2]{#2}

\end{document}